\documentclass[12pt,reqno]{amsart}

\usepackage{amsmath,amsthm,amssymb,comment,fullpage}
\usepackage{braket}
\usepackage{mathtools}

\usepackage{caption}
\usepackage{times}
\usepackage[T1]{fontenc}
\usepackage{mathrsfs}
\usepackage{latexsym}
\usepackage[dvips]{graphics}
\usepackage{epsfig}
\usepackage{amsmath,amsfonts,amsthm,amssymb,amscd}
\input amssym.def
\input amssym.tex
\usepackage{color}
\usepackage{hyperref}
\usepackage{url}
\newcommand{\bburl}[1]{\textcolor{blue}{\url{#1}}}

\usepackage{tikz}
\usepackage{tkz-tab}
\usepackage{tkz-graph}
\usetikzlibrary{shapes.geometric,positioning}

\newcommand{\burl}[1]{\textcolor{blue}{\url{#1}}}

\numberwithin{equation}{section}

\newtheorem{thm}{Theorem}[section]

\newtheorem{lem}[thm]{Lemma}

\theoremstyle{plain}

\newtheorem{lemma}[thm]{Lemma}

\newtheorem{theorem}[thm]{Theorem}
\newtheorem{conjecture}[thm]{Conjecture}

\newtheorem{rem}[thm]{Remark}
\newtheorem{remark}[thm]{Remark}

\newcommand\be{\begin{equation}}
\newcommand\ee{\end{equation}}
\newcommand\bee{\begin{equation*}}
\newcommand\eee{\end{equation*}}
\newcommand\bea{\begin{eqnarray}}
\newcommand\eea{\end{eqnarray}}
\newcommand\beae{\begin{eqnarray*}}
\newcommand\eeae{\end{eqnarray*}}
\newcommand\bi{\begin{itemize}}
\newcommand\ei{\end{itemize}}
\newcommand\ben{\begin{enumerate}}
\newcommand\een{\end{enumerate}}
\newcommand\bc{\begin{center}}
\newcommand\ec{\end{center}}
\newcommand\ba{\begin{array}}
\newcommand\ea{\end{array}}

\newcommand{\N}{\mathbb{N}}

\newcommand\frakfamily{\usefont{U}{yfrak}{m}{n}}
\DeclareTextFontCommand{\textfrak}{\frakfamily}

\newcommand{\hr}[1]{\href{#1}{\url{#1}}}

\newcommand\abs[1]{\left|#1\right|}

\newcommand{\floor}[1]{\left\lfloor#1\right\rfloor}

\newcommand{\E}[1]{\mathbb{E}[#1]}

\title{Limiting Distributions in Generalized Zeckendorf Decompositions}

\author{Alexandre Gueganic}
\email{\textcolor{blue}{\href{mailto:ag15@williams.edu}{ag15@williams.edu}}}
\address{Department of Mathematics and Statistics, Williams College, Williamstown, MA 01267}

\author{Granger Carty}
\email{\textcolor{blue}{\href{mailto:Granger.L.Carty@williams.edu}{Granger.L.Carty@williams.edu}}}
\address{Department of Mathematics and Statistics, Williams College, Williamstown, MA 01267}

\author{Yujin H. Kim}
\email{\textcolor{blue}
{\href{mailto:yujin.kim@columbia.edu}
{yujin.kim@columbia.edu}}}
\address{Department of Mathematics, Columbia University, New York, NY 10027}

\author{Steven J. Miller}
\email{\textcolor{blue}{\href{mailto:sjm1@williams.edu}{sjm1@williams.edu}}}
\address{Department of Mathematics and Statistics, Williams College, Williamstown, MA 01267}

\author{Alina Shubina}
\email{\textcolor{blue}{\href{mailto:as31@williams.edu}{as31@williams.edu}}}
\address{Department of Mathematics and Statistics, Williams College, Williamstown, MA 01267}

\author{Shannon Sweitzer}
\email{\textcolor{blue}{\href{mailto:sswei001@ucr.edu}{sswei001@ucr.edu}}}
\address{Department of Mathematics, University California, Riverside, CA 92521}

\author{Eric Winsor}
\email{\textcolor{blue}{\href{mailto:rcwnsr@umich.edu}{rcwnsr@umich.edu}}}
\address{Department of Mathematics, University of Michigan, Ann Arbor, MI 48109 }

\author{Jianing Yang}
\email{\textcolor{blue}{\href{mailto:jyang@colby.edu}{jyang@colby.edu}}}
\address{Department of Mathematics and Statistics, Colby College, Waterville, ME 04901}

\thanks{This research was supported by NSF grants DMS1265673 and DMS1347804 and Williams College, and by the Finnerty Fund.}


\date{\today}

\begin{document}
\maketitle

\begin{abstract} An equivalent definition of the Fibonacci numbers is that they are the unique sequence such that every integer can be written uniquely as a sum of non-adjacent terms. We can view this as we have bins of length 1, we can take at most one element from a bin, and if we choose an element from a bin we cannot take one from a neighboring bin. We generalize to allowing bins of varying length and restrictions as to how many elements may be used in a decomposition. We derive conditions on when the resulting sequences have uniqueness of decomposition, and (similar to the Fibonacci case) when the number of summands converges to a Gaussian; the main tool in the proofs here is the Lyaponuv Central Limit Theorem.
\end{abstract}

\tableofcontents

\section{Introduction}

\subsection{Preliminaries}
The Fibonacci numbers are normally defined by the recurrence $F_{n+1} = F_n + F_{n-1}$, with, of course, two initial conditions. If we take $F_1 = 1$ and $F_2 = 2$ one of many properties is Zeckendorf's Theorem \cite{Ze}: Every positive integer can be written uniquely as a sum of non-adjacent Fibonacci numbers. Interestingly, this is an equivalent definition of the Fibonaccis; explicitly, they are the unique sequence of numbers such that every integer can be written as a sum of non-adjacent elements of the set. This correspondence has led to numerous papers investigating connections between sequences and decomposition laws, and properties of the decompositions (such as on average how many summands are needed, what is the distribution of gaps between summands, what is the longest gap between summands). We often refer to these as generalized Zeckendorf decompositions or legal decompositions for the given law; for 2019 we have $$2019  \ = \  1597 + 377 + 34 + 8 + 3 \ = \ F_{16} + F_{13} + F_{8} + F_5 + F_3.$$

There is now an extensive literature on the subject; see for example \cite{Al, BBGILMT, BILMT, Br1, Br2, CFHMN1, CFHMN2, CFHMNPX, Day, DDKMMV, KKMW, Fr, GTNP, Ha, Ho, HW, Ke, KKMW, MW1, MW2, Ste1, Ste2}. Of these, the most relevant for our investigations below is \cite{CFHMN1}. There the authors generalize the Fibonacci decomposition law by adopting a binning perspective. Explicitly, fix positive integers $s$ and $b$. The $(s, b)$-Generacci sequence is defined as follows. Consider a series of bins of length $b$. We can choose at most one element from a bin, and if we choose an element we cannot take an element from any of the $s$ bins immediately to the left (and thus we also cannot take an element from any of the $s$ bins immediately to the right). The Fibonaccis correspond to the case $s=b=1$, and choosing the appropriate initial conditions always yields unique decomposition. For example, the $(1,2)$-Generacci sequence begins \begin{equation} \underbracket{\ 1,\ 2\ }_{{\rm Bin\ 1}}\ ,\ \underbracket{\    3,\  4 \ }_{{\rm Bin \ 2}}\ ,\ \underbracket{\      5,\ 8 \ }_{{\rm Bin \ 3}}\ ,\ \underbracket{\     11,\  16 \ }_{{\rm Bin \ 4}}\ ,\ \underbracket{\     21,\  32 \ }_{{\rm Bin  \ 5}}\ ,\ \underbracket{\     43,\  64 \ }_{{\rm Bin \ 6 }}\ , \ \underbracket{\      85,\ 128 \ }_{{\rm Bin \ 7}}\ ,\ \underbracket{\     171,\  256 \ }_{{\rm Bin \ 8}}\ ,\ \ldots.\end{equation}

In previous works all bins had the same length, and a legal decomposition could have at most one element from a bin. We extend these results by now letting the $n$\textsuperscript{{\rm th}} bin have length $b_n \ge 1$, for each $n$.
Furthermore, we choose a set
$A_n \subset \{0, 1, 2, \dots, b_n\}$, which is the set of the number of allowable elements we can choose from the $n$\textsuperscript{{\rm th}} bin in our decomposition.
Finally, we select an adjacency number $a$ such that we cannot take elements from two different bins unless there are at least $a$ bins between them.
Thus, if $b_8 = 5$, $A_8 = \{0, 1, 3\}$, and $a=2$, then we may take $0$, $1$ or $3$ elements from the eighth bin (which has length $5$); if we do take an element from the eighth bin, then we may not take any elements from the sixth, seventh, ninth or tenth bins in our decomposition. We construct the sequence as follows. We set $1$ as the first element of the first bin (we choose $1$ and not $0$ to retain the possibility of having unique decompositions). If we have constructed the first $k$ elements, the next term in the sequence is the least integer which cannot be obtained by our construction rule. We refer to these as a $(\{b_n\}, \{A_n\}, a)$-Sequence; the Fibonacci sequence is $b_n =  1$, $A_n = \{0,1\}$ and $a=1$.

\subsection{Results}


In Section \ref{unique_decomp} we study sequences with no adjacency condition (i.e., $(\{b_n\}, \{A_n\}, 0)$-Sequences), and exactly determine when these sequences give us unique decomposition of the positive integers (see \cite{CHHMPV} for conditions on when generalized Zeckendorf decompositions have the minimal number of summands among all decompositions). In particular, we prove the following.

\begin{theorem}\label{thm:uniquenessnoadjcondition} A $(\{b_n\}, \{A_n\}, 0)$-Sequence has uniqueness of decomposition (i.e., there is a unique legal decomposition for each positive integer) if and only if for every positive $n$ we have \begin{equation} A_n \ \in\ \left\{\left\{0,1\right\},\ \left\{0,1,\ldots,b_n-1\right\}, \ \left\{0,1,\ldots,b_n\right\}\right\}.\end{equation}
\end{theorem}

In Section $3$ we establish the following Lyapunov central limit type theorems associated to certain $({b_n}, {A_n}, 0)$-Sequences. These results are similar to those from earlier work on Zeckendorf decompositions. Lekkerkerker \cite{Le} proved that the average number of summands in a Zeckendorf decomposition for integers in $[F_n, F_{n+1}]$ tends to $\frac{n}{\varphi^2+1}$, where $\varphi = \frac{1+ \sqrt{5}}{2}$; others (see for example \cite{KKMW}) extended this result to prove that as $n \to \infty$, the distribution of the number of summands in the Zeckendorf decomposition for integers in $[F_n, F_{n+1}]$ is Gaussian. In Section \ref{proof_01A_n} we prove a similar result for our sequences, using Lyapunov's Central Limit Theorem (see Theorem \ref{lyapunovCLT}).

\begin{theorem}\label{01A_n} Consider a $(\{b_n\}, \{0, 1\}, 0)$-Sequence. For an integer $x$, let $Y_n(x) = 1$ if an element of the $n$\textsuperscript{{\rm th}} bin appears in $x$'s decomposition, and $Y_n(x) = 0$ otherwise; thus, if the largest summand in $x$'s decomposition is from bin $N$ then the total number of summands in this decomposition is $Y_1(x) + \cdots + Y_N(x)$. If $\sum_{n=1}^{\infty} 1/b_n$ diverges, then the distribution of the number of summands of integers whose largest summand is in bin $N$ converges to a Gaussian in the sense of Lyapunov as $N\to \infty$. \label{varybinsize}
\end{theorem}

In Section \ref{proof_3.thm_bnA}, we relax our assumptions to allow multiple summands from each bin, and let $A_n$ vary with $n$; we examine how the conditions for Gaussianity change given this generalization in the following two theorems.

\begin{theorem}\label{3.thm_bnA} Consider a $(\{b_n\}, \{A\}, 0)$-Sequence, where each $A_n = A \subseteq \{0, 1, \dots, b\}$ with $b \le \min(\{b_n\})$. Let $\{Y_n\}$ be the sequence of independent random variables representing the number of summands chosen from each bin. Thus if the largest summand of the decomposition of an integer $x$ is from bin $N$, then the total number of summands in this decomposition is $Y_1(x) + \dots + Y_N(x)$.
If the growth of $\{b_n\}$ is slower than $n^{\frac{1}{m-m^{'}}}$, where $m = \max(A)$ and $m^{'} = \max(A-\{m\})$, then the distribution of the number of summands of integers whose largest summand is in bin $N$ converges to a Gaussian distribution in the Lyapunov sense as $N \to \infty$.
\end{theorem}

\begin{theorem} \label{3.thm_bnAn}
Consider a $(\{b_n\}, \{A_n\}, 0)$-Sequence, where for all $n\in \N$, $b_n = n$, and  $A_n \in \{ \{0,\dots,n-1\},\{0,\dots,n\} \}$.
Let $\{Y_n\}$ be the sequence of independent random variables representing the number of summands chosen from each bin.
For any integer choice of $\delta >0$, the distribution of the number of summands satisfies the Lyapunov Central Limit Theorem, and thus converges to a Gaussian distribution as $N \to \infty$.
\end{theorem}

We conclude in Section $4$ with a discussion of related lines for future research.

\section{Uniqueness of Decomposition with no Adjacency Condition}\label{unique_decomp}



We consider an arbitrary $(\{b_n\}, \{A_n\}, 0)$-Sequence; as $a=0$ there is no adjacency restriction.
We categorize what choices of the sequence $A_n$ give uniqueness of decomposition for the resulting generalized Zeckendorf decompositions. We usually require that $0$ and $1$ are in each $A_n$,
i.e, $\{0,1\}\subset A_n$, to ensure that our original construction creates a sequence where every positive integer has a decomposition.%
\footnote{If $A_n$ does not contain 0, then any decomposition must include an element of bin $n$, which forces the sum of a decomposition to be at least that of the minimal element of $A_n$, destroying our hopes of having either uniqueness or a decomposition for every positive integer. Note that if $A_n$ does not contain 1, zeroes can be added to bin $n$ so that way are able to pick any one particular element, though at the cost of uniqueness. For example, if we want to use just one element of bin $n$, and $A_n = \{k, k+1, \dots, b\}$, then we can place $k-1$ zeros in $b_n$. }
In Section \ref{g-nary}, we consider a scenario where $A_n = \{0,2\}$, but we do not require our sequence to generate the positive integers.

To understand the proof of Theorem 1.1, we use the following intuition. In our construction of a generalized Zeckendorf sequence, we ensure that each integer is generated by the construction ``in order'', that is, if we look at the first $k$ terms of our $(\{b_n\}, \{A_n\}, 0)$-Sequence, we will see that a consecutive block of positive integers is uniquely decomposable using these terms. When we allow $A_n$ to violate the conditions of Theorem 1.1, the first $k$ terms of our sequence no longer generate a consecutive block; the decomposable integers form multiple disconnected blocks. The block containing $1$ continues to grow as we add terms to our sequence and eventually meets another block, causing a failure of uniqueness of decomposition for some integer.

\begin{lemma} \label{interval generation}
Fix a $(b_n, A_n, 0)-$Sequence, and an integer $n_0 \geq 2$. Suppose that the set of integers generated by the first $n_0-1$ bins is the set $\left\{1,\ldots,k\right\}$. Then all future terms of our sequence are divisible by $k+1$.
\end{lemma}

\begin{proof}
Note that the first term in bin $n_0$ must be $k+1$. The terms in the first $n_0-1$ bins can form any sum from $1$ to $k$, and thus as we have no adjacency conditions, if we can represent a number $x$ using numbers from bin $n_0$ and on, we can also obtain $x+1, x+2, \dots, x+k$. Thus once we add a multiple $\alpha(k+1)$ of $k+1$, there is no need to add $\alpha(k+1) + \beta$ for any $\beta \in \{1, \dots, k\}$, and therefore the next possible term in our sequence is $(\alpha+1)(k+1)$. Continuing we see that all the numbers added are multiples of $k+1$, proving the claim.
\end{proof}


For example, consider the sequence with $b_n=n+1,\ A_n=\{0,1\}$:
\begin{equation} \underbracket{\ 1,\ 2\ }_{{\rm Bin\ 1}}\ ,\ \underbracket{\ 3,\ 6,\ 9\ }_{{\rm Bin \ 2}}\ ,\ \underbracket{\ 12,\ 24,\ 36,\ 48\ }_{{\rm Bin \ 3}}\ ,\ \underbracket{\  60,\ 120,\  \ldots}_{{\rm Bin \ 4}}\ ,\ \ldots.\end{equation}

Letting $n_0=2$ we find $k=11$ (i.e., the first two bins allow us to obtain precisely the integers from 1 to 11), and see that any legal combination of terms outside the first two bins is a multiple of $12$. \\

\begin{proof}[Proof of Theorem \ref{thm:uniquenessnoadjcondition}]
We want to show that a $(\{b_n\}, \{A_n\}, 0)$-Sequence has uniqueness of decomposition if and only if all $A_n$ are in the form of $\left\{0,1\right\},\ \left\{0,1,\ldots,b_n-1\right\} \text{ or } \left\{0,1,\ldots,b_n\right\}$.\\

To reduce the cases that we need to discuss, we assume that the first $n_0 - 1$ bins have $A_n$'s that satisfy the condition and the set of legal sums from these bins form the interval $\left\{1, \ldots k \right\}$, where each element has unique decomposition. Then by Lemma \ref{interval generation}, we have that all following terms of the sequence are divisible by $k+1$. Therefore, we can take the subsequence of our original sequence starting from the $n_0$\textsuperscript{th} bin to be our new sequence, and divide every term by $k+1$. For notational convenience we denote $A_{n_0}, b_{n_0}$ of the original sequence as $A_1, b_1$ of the new sequence which we now analyze.\\

We first show that if $A_1$ satisfies one of the conditions for which we claim uniqueness holds, then it yields intervals of integers, so by induction the first $n$ bins of the sequence always yield an interval for any $n\in \mathbb{N}$. Since every element of this interval has unique decomposition, we can prove the backwards direction of Theorem \ref{thm:uniquenessnoadjcondition}. Next we consider the case where the new sequence has $A_1$ outside of our stated set. We are then able to show that uniqueness fails in such sequences, so only the options stated in Theorem \ref{thm:uniquenessnoadjcondition} give uniqueness, therefore proving the forwards direction of the theorem.\\

We now consider each case for uniqueness.\\

\noindent \textbf{Case I:} $\mathbf{A_1 = \left\{0,1\right\}}$.
Fix $b_1$ and let $A_1 = \left\{0,1\right\}$. Then the first $b_1$ terms of our sequence are $1, 2, \ldots, b_1$. The integers generated by this bin form the set $S = \{ 1, 2, \ldots, b_1 \}$, which is an interval of integers. Since $A_1 = \{0,1\}$, and each element of $S$ must be written as a sum of elements in $b_1$, we clearly have unique decomposition.

\ \\

\noindent \textbf{Case II:} $\mathbf{A_1 = \left\{0,1,\ldots,b_1\right\}}$.
Fix $b_1$ and let $A_1 = \left\{0,1,\ldots,b_1\right\}$. Then the first $b_1$ terms of our sequence are $1,2,4,\ldots,2^{b_1-1}$. The integers generated by this bin form the set $S = \{ 1,2,\ldots,2^{b_1}-1 \}$, which is an interval of integers. Because binary decomposition of the integers is unique, we have unique decomposition.

\ \\

\noindent \textbf{Case III:} $\mathbf{A_1 = \left\{0,1,\ldots,b_1-1\right\}}$.
Fix $b_1$ and let $A_1 = \left\{0,1,\ldots,b_1-1\right\}$. Then the first $b_1$ terms of our sequence are $1,2,4,\ldots,2^{b_1-1}$. The integers generated by this bin form the set $S = \{ 1,2,\ldots,2^{b_1}-2 \}$, which is an interval of integers. We also note that this choice of $A_1$ gives unique decomposition, for the same reason as Case II.

\ \\

We have now explicitly analyzed the cases we claim give uniqueness and have shown that they yield intervals of integers. We are thus able to reduce to the cases where $A_1$ is not in the given set. We split non-uniqueness of these other choices of $A_1$ into several cases.

\ \\
\noindent \textbf{Case I:} $\mathbf{\left\{0,1,\ldots,k\right\}\subset A_1,\ {\rm with}\ k+1\notin A_1\ {\rm and}\ 2\leq k \leq b_1-2}$. 0
Because we have full freedom with the first $k$ elements of $b_1$, we have $1,2,4,\ldots,2^k$ as the first $k+1$ elements of this bin. Arguing as before, we also have that the $(k+2)$\textsuperscript{{\rm nd}} element of our bin must be $2^{k+1}-1$. We must use this term to form larger integers, so we are left with only $k-1$ terms to work with, meaning we can form all integers up to but not including $2^{k+1}-1+2^k-2+1 = 2^{k+1}+2^k-2$. Thus, this is the $(k+3)$\textsuperscript{{\rm rd}} element of our sequence (it will not matter whether this is in the first or second bin). We note that we can decompose $2^{k+1}+2^k-1$ as
\begin{align}
\left(2^{k+1}-1\right)+2^k\ =\ 2^{k+1}+2^k-1\ =\ \left(2^{k+1}+2^k-2\right)+1,
\end{align}
so uniqueness fails.

\ \\

\noindent \textbf{Case II:} $\mathbf{\left\{0,1\right\} \subsetneq A_1,\ 2\notin A_1}$. Pick $k := \inf \{x \in A_1 : x >1 \}$. This is the case where there is a gap in $A_1$. Since we are only allowed to choose $0, 1$ or at least $k$ elements from a bin, the first $k$ terms of the sequence are going to be $1,\ldots,k$. Since $k\geq 3$, $\sum_{m = 1}^k m = \frac{k(k+1)}{2} > k+2$, so the $(k+1)^{\mathrm{st}}$ and the $(k+2)^{\mathrm{nd}}$ terms are $k+1$ and $k+2$, respectively.\\
If we have $b_1\geq k+2$ for the first bin, then
\begin{align}
\begin{cases}
\displaystyle\sum_{m = 1}^{k/2}m+\sum_{m = k/2+3}^{k+2}m\ =\ \sum_{m = 2}^{k+1}m\ =\ \frac{k(k+3)}{2}\ &\text{when }k\text{ is even}\\
\displaystyle\sum_{m = 1}^{(k-1)/2}m+\sum_{m = (k+5)/2}^{k+2}m\ =\ \sum_{m = 2}^k m+(k+2)\ =\ \frac{(k-1)(k+2)}{2}+k+2\ &\text{when }k\text{ is odd},
\end{cases}
\end{align}
and we lose uniqueness of decomposition. Therefore, we only need to consider the cases where $b_1 = k$ or $k+1$. As the two follow similarly, we only provide the details for the first.

\ \\

\textbf{Subcase (i): }$\mathbf{b_1 = k}$. As $b_1 = k$, the sum of terms from the first bin is $\frac{k(k+1)}{2}$. As argued before, all multiples of $k+1$ less than this sum, including $\frac{k-1}{2}(k+1)$, can be expressed as a legal sum of terms not in the first bin. Therefore, when $k$ is odd,
\begin{align}
\frac{k(k+1)}{2}\ =\ \frac{k-1}{2}(k+1)+\frac{k+1}{2},
\end{align} where $\frac{k+1}{2}$ is a term in the first bin. We lose uniqueness of decomposition. When $k$ is even, $\frac{k(k+1)}{2}$ is not in the sequence and the next term is $\frac{k(k+1)}{2}+1$. Then we can decompose $(k+1)+\frac{k(k+1)}{2}$ two ways:
\begin{align}
(k+1)+\frac{k(k+1)}{2}\ =\ \left(\frac{k(k+1)}{2}+1\right)+k,
\end{align} where $k+1$ and $k$ are terms of the sequence. We lose uniqueness of decomposition.

\ \\

\textbf{Subcase (ii): }$\mathbf{b_1 = k+1}$. A similar argument holds on losing uniqueness of decomposition.
\end{proof}

\section{Gaussianity of Number of Summands: $a=0$}


Now that we have exactly determined the decomposition rules which yield sequences giving rise to unique decomposition of integers in the $a=0$ case, we investigate the Gaussianity of the distribution of the average number of summands in these decompositions. The following result (see \cite{Bi}) is a key ingredient in several proofs in this section.

\begin{theorem}[Lyapunov Central Limit Theorem]\label{lyapunovCLT}
Let $\{Y_1, Y_2, \dots\}$ be a sequence of independent random variables, each with finite expected value $\mu_i$ and variance $\sigma_{i}^2$. Let $s_n^2 := \sum_{i=1}^{n} \sigma_{i}^2$. If there exists a $\delta>0$ such that $\lim_{n \to \infty} \frac{1}{s_{n}^{2+\delta}}\sum_{i=1}^{n}\mathbb{E}[|Y_{i}-\mu_{i}|^{2+\delta}] = 0$, then $\frac{1}{s_n} \sum_{i=1}^{n} (Y_i - \mu_i)$ converges in distribution to the standard normal as $n \to \infty$.
\end{theorem}

We use the following standard notation below. We write $f(x) = \Theta(g(x))$ if there exist positive constants $C_1, C_2$ such that for all $x$ sufficiently large we have
\begin{equation}
0 \ < \ C_1 g(x) \ \le \ f(x) \ \le \ C_2 g(x).
\end{equation}

\subsection{At most one summand per bin} \label{proof_01A_n}
We begin by proving Theorem \ref{01A_n}, which concerns sequences with variable bin sizes, $A_n=\{0,1\}$, and no adjacency condition.

\begin{proof}[Proof of Theorem \ref{01A_n}]
For $n<N$, we have $b_n+1$ options for the $n$\textsuperscript{th} bin: we have no element or exactly one of the $b_n$ terms. Each of these choices is equally likely, and thus $P(Y_n = 0)=\frac{1}{b_{n}+1}$ and $P(Y_{n}=1) = \frac{b_n}{b_n +1}$. Therefore the expected value for $Y_n$ (and $Y_n^2$ as $Y_n = Y_n^2)$ is
\begin{equation}
\mu_n\ :=\ \mathbb{E}[Y_{n}]\ =\ \frac{b_n}{b_{n}+1} \ = \ \mathbb{E}[Y_n^2],
\end{equation}
and its variance is
\begin{align}
\sigma_{n}^2 \ :=\ \mathbb{E}[Y_{n}^2] -\mathbb{E}[Y_n]^2\ =\  \frac{b_n}{b_n +1}-\left(\frac{b_n}{b_n +1}\right)^2\ =\ \frac{b_n}{(b_n +1)^2}.
\end{align}
Let $s_N^2 := \sum_{n=1}^{N-1} \sigma_{n}^2 = \sum_{n=1}^{N-1} b_n/(b_n +1)^2$. We now apply the Lyapunov Central Limit Theorem. Note
\begin{align}
\mathbb{E}[|Y_n-\mu_n|^{2+\delta}] & \ = \  \frac{b_n}{b_n +1}\left(\frac{1}{b_n +1}\right)^{2+\delta} + \frac{1}{b_n +1}\left(\frac{b_n}{b_n +1}\right)^{2+\delta} \nonumber\\
& \ = \  \frac{b_n}{(b_n +1)^2}\frac{1+b_n^{1+\delta}}{(b_n +1)^{1+\delta}}\ <\ \frac{b_n}{(b_n +1)^2}.
\end{align}
Define $\rho_n^{2+\delta} := \E{|Y_n-\mu_n|^{2+\delta}}$ and $e_N := \sum_{n=1}^{N}\rho_n^{2+\delta}$. Then
\begin{align}
e_N \ =\ \sum_{n=1}^{N-1}\mathbb{E}[|Y_n-\mu_n|^{2+\delta}]\ <\ \sum_{n=1}^{N-1} \frac{b_n}{(b_n +1)^2}\ =\ s_N^2.
\end{align}
We note that $\sigma_n^2$ is asymptotically similar to $1/b_n$ (i.e., $1/b_n \ll \sigma_n \ll 1/b_n)$, so $\left\{s_N^2\right\}$ converges if and only if $\sum_{n=1}^N 1/b_n$ converges.

Suppose $\sum_{n=1}^N 1/b_n$ diverges. Then $s_N^2$ diverges, and for all $\delta>0$
\begin{align}
\lim_{N \to \infty} \left(\frac{e_N}{s_N^{2+\delta}}\right)^2 \ <\ \lim_{N \to \infty} \frac{(s_N^2)^2}{(s_N^2)^{2+\delta}}\ =\ \lim_{N \to \infty} \frac{1}{(s_N^2)^{\delta}}\ =\ 0
\end{align}
(the limit tends to zero as we are assuming the sum of the reciprocals of $b_n$ diverges, and thus $n_n$ must tend to infinity). Thus, the Lyapunov condition is satisfied,  and by Theorem \ref{lyapunovCLT} the distribution of number of summands, $\frac{1}{N}\sum_{i=1}^N Y_i$, converges to a Gaussian in the sense of Lyapunov.
\end{proof}

\begin{remark}
If $\sum_{i=1}^n 1/b_i$ converges, then the denominator of the Lyapunov limit converges to a finite limit. Furthermore, the numerator is nonzero, so the limit is nonzero. Thus, the Lyapunov condition fails if $\sum_{i=1}^n 1/b_i$ converges. While this does not prove that the distribution of the number of summands does not approach a Gaussian distribution, it provides some evidence against this behavior.
\end{remark}

\subsection{
Multiple summands per bin
} \label{proof_3.thm_bnA}



We now prove Theorem \ref{3.thm_bnA}.

\begin{proof}[Proof of Theorem \ref{3.thm_bnA}]
Assume $|A| \geq 2$.
We begin in a similar manner as Theorem \ref{varybinsize} by noting that the probability of choosing exactly $i$ summands from the $n$\textsuperscript{{\rm th}} bin is
\begin{align}
P(Y_n =i)\ =\ \frac{\binom{b_n}{i}}{\sum_{t\in A}\binom{b_n}{t}},
\end{align}
and the expectated values of $Y_n$ and $Y_n^2$ are
\begin{eqnarray}
\mathbb{E}[Y_n] \ =\  \frac{\sum_{t\in A}t\binom{b_n}{t}}{\sum_{t\in A}\binom{b_n}{t}}, \ \ \ \mathbb{E}[Y_n^2]  \ =\  \frac{\sum_{t\in A}t^2\binom{b_n}{t}}{\sum_{t\in A}\binom{b_n}{t}}.
\end{eqnarray}
Then
\begin{align}
\sigma_n^2& \ = \  \mathbb{E}[Y_n^2]-\mathbb{E}[Y_n]^2\nonumber\\
& \ = \  \frac{\left(\sum_{t\in A}t^2\binom{b_n}{t}\right)\left(\sum_{t\in A}\binom{b_n}{t}\right)}{\left(\sum_{t\in A}\binom{b_n}{t}\right)^2}-\frac{\left(\sum_{t\in A}t\binom{b_n}{t}\right)^2}{\left(\sum_{t\in A}\binom{b_n}{t}\right)^2}\nonumber\\
& \ = \  \frac{\sum_{i,j\in A}i^2\binom{b_n}{i}\binom{b_n}{j}-\sum_{i,j\in A}ij\binom{b_n}{i}\binom{b_n}{j}}{\left(\sum_{t\in A}\binom{b_n}{t}\right)^2}.
\end{align}
The terms where $i = j$ cancel, so we are left with
\begin{align}
\sigma_n^2 & \ = \
\frac{\sum_{i,j\in A, i\neq j}i^2\binom{b_n}{i}\binom{b_n}{j}-\sum_{i,j\in A, i \neq j}ij\binom{b_n}{i}\binom{b_n}{j}}{\left(\sum_{t\in A}\binom{b_n}{t}\right)^2}
\ = \
\frac{\sum_{i,j\in A,i\neq j}(i-j)^2\binom{b_n}{i}\binom{b_n}{j}}{2\left(\sum_{t\in A}\binom{b_n}{t}\right)^2}. \label{3.sigman2}
\end{align}
Define $\rho_n^{2+\delta} := \mathbb{E}\left[\left|Y_n-\mu_n\right|^{2+\delta}\right]$. We find that
\begin{align}
\rho_n^{2+\delta} & \ = \  \sum_{i\in A} \left|i-\frac{\sum_{t\in A}t\binom{b_n}{t}}{\sum_{t\in A}\binom{b_n}{t}}\right|^{2+\delta} \frac{\binom{b_n}{i}}{\sum_{t\in A}\binom{b_n}{t}}
\nonumber\\
& \ = \  \frac{\sum_{i\in A} \binom{b_n}{i} \left|i\sum_{t\in A} \binom{b_n}{t} -\sum_{t\in A} t\binom{b_n}{t}\right|^{2+\delta}}{\left(\sum_{t\in A} \binom{b_n}{t}\right)^{3+\delta}}
\nonumber\\
& \ = \  \frac{\sum_{i\in A} \binom{b_n}{i} \left|\sum_{t\in A} \left(i-t\right)\binom{b_n}{t} \right|^{2+\delta}}{\left(\sum_{t\in A} \binom{b_n}{t}\right)^{3+\delta}}. \label{3.rho_n}
\end{align}
We now find asymptotics for $\sigma_n^2$ and $\rho_n^{2+\delta}$. We first note that $\binom{b_n}{t} = \Theta(b_n^t)$; we do not need to have a $t$ subscript on the $\Theta$ relation as $t \le b$ and $b$ is fixed. Therefore
\begin{align}
\left(\sum_{t\in A}\binom{b_n}{t}\right)^2 \ =\ \left(\sum_{t\in A} \Theta\left(b_n^t\right)\right)^2\ =\ \Theta\left(b_n^{m}\right)^2\ =\ \Theta\left(b_n^{2m}\right).
\end{align}
We also note that
\begin{align}
\sum_{i,j\in A, i\neq j}(i-j)^2\binom{b_n}{i}\binom{b_n}{j}\ =\ \sum_{i,j\in A, i\neq j} \Theta\left(b_n^i b_n^j\right)\ =\ \Theta\left(b_n^{m+m^\prime}\right).
\end{align}
Therefore
\begin{align}
\sigma_n^2\ =\ \frac{\Theta(b_n^{m+m^{'}})}{\Theta(b_n^{2m})}\ =\ \Theta\left(\frac{1}{b_n^{m-m^{'}}}\right).
\end{align}
Similarly, for $\rho_n^{2+\delta}$ we have
\begin{align}
\rho_n^{2+\delta} \ = \  \frac{\Theta(b_n^{(2+\delta)m}b_n^{m'})}{\Theta(b_n^{(3+\delta)m})}\ =\ \Theta(b_n^{m'-m}).
\end{align}
Thus
\begin{align}
\Theta(\rho_n^{2+\delta})\ =\ \Theta(\sigma_n^2).
\end{align}
Now let $r_N^{2+\delta} := \sum_{n=1}^N \rho_n^{2+\delta}$, and $s_N^2 := \sum_{n=1}^N\sigma_n^2$. We consider the Lyapunov limit $\lim_{N \to \infty} r_N^{2+\delta}/s_N^{2+\delta}$. We have
\begin{align}
\lim_{N\to \infty}\left(\frac{r_N^{2+\delta}}{s_N^{2+\delta}}\right)^2\ =\ \lim_{N\to \infty}\frac{(r_N^{2+\delta})^2}{(s_N^2)^{2+\delta}} \ = \ \lim_{N\to \infty}\frac{\Theta(s_N^2)^2}{(s_N^2)^{2+\delta}}\  =\ \lim_{N\to \infty}\frac{1}{\Theta(s_N^2)^\delta}.
\end{align}
If the bin size $b_n$ grows slower than $n^{\frac{1}{m-m^{'}}}$, then $\sum_{n=1}^{N}\Theta\left(1/(b_n^{m-m^{'}})\right) \to \infty$, and thus $s_N \to \infty$. Thus, the above limit tends to $0$ and the Lyapunov condition is satisfied for any $\delta > 0$. So we conclude that the distribution of the number of summands converges to a Gaussian distribution as $N \to \infty$.
\end{proof}




We now prove Theorem \ref{3.thm_bnAn}.

\begin{proof}[Proof of Theorem \ref{3.thm_bnAn}]
We will prove the case $A_n = \{0,\dots, n\}$, as the case $A_n = \{0,\dots, n-1\}$ is similar.

Taking $b_n = n$ and $A_n = \{0,\dots, n\}$ in \eqref{3.sigman2} and \eqref{3.rho_n}, we have
\begin{align}
    \rho_n^{2+\delta} &\ = \ \frac{\sum_{i=0}^n \binom{n}{i} \abs{2i-n}^{2+\delta}}{2^{n+\delta+2}} \nonumber \\
    \sigma_n^2  &\ = \   \frac{n}{4}.
\end{align}

From the Lyapunov CLT limit, we seek to show
\begin{align}
\lim_{N \rightarrow \infty} \dfrac{\sum_{n \leq N} \rho_n^{2+\delta}}{\left ( \sum_{n \leq N} \sigma_n^2 \right )^{\frac{2+\delta}{2}} } &\ = \ 0, \ \ \ \text{or equivalently } \lim_{N \rightarrow \infty} \dfrac{\left ( \sum_{n \leq N} \sigma_n^2 \right )^{\frac{2+\delta}{2}} }{\sum_{n \leq N} \rho_n^{2+\delta}} \ = \ \infty. \label{3.2.lyapunov}
\end{align}

Substituting gives, for fixed $N$,
\begin{align}
\dfrac{\left ( \sum_{n \leq N} \sigma_n^2 \right )^{\frac{2+\delta}{2}} }{\sum_{n \leq N} \rho_n^{2+\delta}}  & \ = \
\frac{ c N^{2+\delta}}{\sum_{n\leq N} \dfrac{\sum_{i=0}^n \binom{n}{i} \abs{2i-n}^{2+\delta}}{2^{n+\delta+2}}},
\end{align}
for a constant $c > 0$.
It thus suffices to show that $\sum_{n\leq N} \dfrac{\sum_{i=0}^n \binom{n}{i} \abs{2i-n}^{2+\delta}}{2^{n+\delta+2}} \ = \ O \left ( N^{1+\delta} \right )$, for which it is enough to prove that
\begin{align}
    \frac{\sum_{i=0}^n \binom{n}{i} \abs{2i-n}^{2+\delta}}{2^{n+\delta+2}} \ = \ O \left ( n^{\delta} \right ). \label{abs_exp}
\end{align}
We assume $\delta>0$ is even in this proof for ease of computations. Then
\begin{align}
\sum_{i=0}^n \binom{n}{i} \abs{2i-n}^{2+\delta} &\ = \ \sum_{i=0}^n \binom{n}{i} (n-2i)^{2+\delta} \nonumber \\
    &\ = \ \sum_{i=0}^n \binom{n}{i} \sum_{j=0}^{2+\delta} \binom{2+\delta}{j} (-1)^j n^{2+\delta-j} (2i)^{j} \nonumber \\
    &\ = \ \sum_{j=0}^{2+\delta} (-1)^j \binom{2+\delta}{j} 2^{j} n^{2+\delta -j} \sum_{i=0}^n \binom{n}{i} i^{j}. \label{full_exp}
\end{align}
We wish to show that the $n^{2+\delta}2^n$ and $n^{1+\delta}2^n$ terms in \eqref{full_exp} go to zero, which correspond to the $n^j 2^{n-j}$ and $n^{j-1} 2^{n-j-1}$ terms in $\sum_{i=0}^n \binom{n}{i} i^j$.
We compute  $\sum_{i =0}^n \binom{n}{i} i^{j}$ by noting that the sum represents the number of ways to choose a subset $A \subset \{1, \dots, n\}$ along with an ordered $j$-tuple $(a_1, \dots, a_j)$, where each $a_k \in A$. Alternatively, we could pick our ordered j-tuple $(b_1, \dots, b_j)$ first, so that each $b_k \in \{ 1, \dots , n \}$, and then choose a subset $B \subset \{1, \dots, n \}$ that includes the distinct elements of $\{b_1, \dots, b_j\}$. It is easily checked that these two counting schemes are the same by showing that the set of possible $\left (A, (a_1, \dots, a_j) \right )$ is in bijection with the set of possible $\left ( (b_1,\dots, b_j), B \right )$. Following our second scheme, we note that if all elements in our $j$-tuple are distinct, then there are
\begin{align}
	n(n-1) \cdots (n-j+1) 2^{n-j} \label{all_dist}
\end{align}
ways to pick our tuple and subset. Similarly, if $j-1$ elements in the $j$-tuple are distinct, then we have
\begin{align}
    \binom{n}{j-1} \binom{j-1}{1} \frac{j!}{2} 2^{n-j+1} \ = \ n(n-1)\cdots (n-j+2) (j-1)j 2^{n-j} \label{all1_dist}
\end{align}
ways to choose.
In general, if $j-k$ elements in our tuple are distinct, then there are $O\left ( n^{j-k} 2^{n-j+k} \right )$ ways to choose our tuple and subset. Therefore, the expressions in \eqref{all_dist} and \eqref{all1_dist} make the only contributions to the $n^{j-1} 2^{n-j-1}$ term, while the only contribution to the $n^j 2^{n-j}$ term comes from \eqref{all_dist}. The coefficient of the $n^j 2^{n-j}$ term is simply $1$, and thus from \eqref{full_exp}, the coefficient of $n^{2+\delta}2^n$ is
\begin{align}
    \sum_{j=0}^{2+\delta} (-1)^j \binom{2+\delta}{j} = 0.
\end{align}
Now, the coefficient of the $n^{j-1}$ term from \eqref{all_dist} is $-\sum_{i=0}^{j-1}i  2^{n-j} =-(j-1)j2^{n-j-1}$. The coefficient of the $n^{j-1}$ term from \eqref{all1_dist} is $(j-1)j2^{n-j}$. We add these two expressions together to obtain  $j^2-j$ as the coefficient of $n^{j-1} 2^{n-j-1}$ in $\sum_{i=0}^n \binom{n}{i} i^j$. Again, from \eqref{full_exp}, the coefficient of $n^{1+\delta}2^n$ is
\begin{align}
\frac{1}{2} \sum_{j=0}^{2+\delta} (-1)^j (j^2-j) \binom{2+\delta}{j} .
\end{align}
Note that the above is equal to
\begin{align}
\left. \frac{\mathrm{d}^2}{\mathrm{d} x^2} (1-x)^{2+ \delta} \right|_{x=1} = 0
\end{align}
for $\delta >0$, and thus we are done.
\end{proof}

\begin{conjecture}
    The Lyapunov condition holds for any for any $A_n = \{0, 1, \dots, \floor{n/k}\}$. Numerics suggest this is true.
\end{conjecture}

\section{Future Directions (Higher Dimensional Sequences)}

\subsection{Zeckendorf Involution Tree}

It would be natural after studying bin decompositions to look at 2-dimensional sequences that have similar properties; the Fibonacci quilt \cite{CFHMN2,CFHMNPX} is one such generalization. We could ask many questions, such as: What types of sequence constructions yield unique decomposition of positive integers? How do statistics such as average number of summands change in the two-dimensional case? However, in many cases (including the Fibonacci quilt), seemingly two-dimensional sequences reduce to one-dimensional relations, such as conditions imposed on bins; see \cite{CCGJMSY} for an example that is fundamentally not one-dimensional. As an example, we construct a ``two-dimensional'' sequence of integers, which we call the Zeckendorf tree, as follows.

Let $a_{1,1}=1$. For a term $a_{i,j}$, $i \geq 1$, $1 \leq j \leq i$, $i$ corresponds to the level in the tree in which the term is located, and $j$ is the term's position within the level. The $i$\textsuperscript{{\rm th}} level has precisely $i$ terms. We add an integer to the tree if it is not the sum of terms from nonadjacent levels. As 2 is not the sum of terms of nonadjacent levels, we add it to the tree as the first term of the second level. Similarly, 3 is the second term of the second level. Next, 4 is the first term of the third level, but 5 can be represented as $4+1$, a sum of terms from nonadjacent levels. So 6 is the next term. We continue this process indefinitely to construct the Zeckendorf tree.

\begin{center}
\scalebox{.8}{\includegraphics{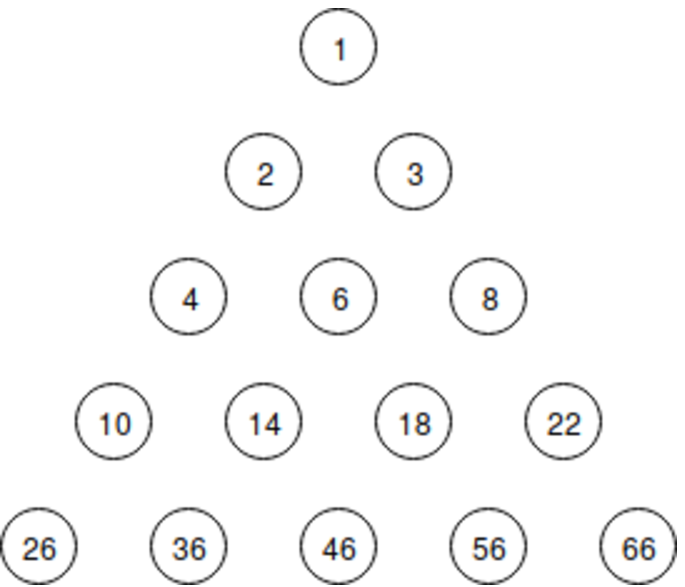}}
\end{center}

Interestingly, the left diagonal of the tree $1, 2, 4, 10, 26,...$ is the sequence of involutions on $i$ letters, also known as the Telephone Numbers. These diagonal terms are defined by the recurrence relation $a_1  = 1, a_2=2$, and $a_n = a_{n-1} + (n-1)a_{n-2}$ for $n>2$.

The recurrence relation for the terms of the tree is given by
\begin{equation}
a_{i,j} \ =\
    \begin{cases}
         a_{i, j-1}+a_{i-1, 0},         &j > 1\\
         a_{i-1, i-1} + a_{i, 0},     &j = 1.
    \end{cases}
\end{equation}

Using techniques similar to those of the proof of Zeckendorf's theorem, one can show that every positive integer $n$ can expressed uniquely as a sum of terms from nonadjacent levels of the Zeckendorf tree. However, while the recurrence relation for the terms of the tree seems to depend both on $i$ and $j$, the tree can be described one-dimensionally using a condition on bins: Let $b_i = i$ be the size of the $i$\textsuperscript{{\rm th}} bin. Then the Zeckendorf tree sequence is the unique sequence constructed by disallowing summands from adjacent bins.

Variations of the Zeckendorf tree retain their two-dimensional nature, but do not always retain uniqueness of decomposition. For example, consider the following tree.

\begin{center}
\scalebox{.8}{\includegraphics{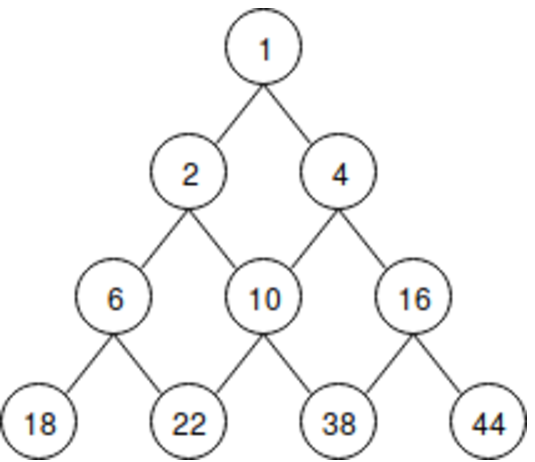}}
\end{center}

We begin the first row with the number $1$ for uniqueness reasons. We construct the sequence using the rule that a term is included if it cannot be composed of summands that are linked in an upwards chain. For example, we do not include $30$ because $30 = 22+6+2$, all of which are linked in an upwards chain.

Example: For $30$, we have

\begin{center}
\scalebox{.8}{\includegraphics{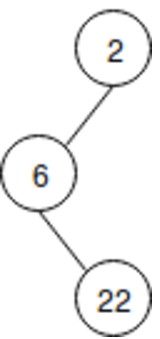}}
\end{center}

However, we do include $38$ because we cannot construct it using such a chain (note we cannot get from 22 to 16). While this sequence cannot be reduced to a condition on bins, it does not have uniqueness of decomposition (for example, $48 = 44+4$ and $48 = 38+10$). We can still prove Gaussianity for the distribution of the number of summands; see \cite{CCGJMSY} for details (as well as extensions to $d$-dimensions).

\subsection{Uniqueness of decomposition in $g$-nary sequences}\label{g-nary} 


We explore another generalization of Zeckendorf sequences: a class of sequences that we call \emph{$g$-nary sequences}. These sequences are quite different from $(\{b_n\}, \{A_n\}, 0)$-Sequences in that they are no longer constrained by the requirement to represent every positive integer. We characterize $g$-nary sequences which give a unique decomposition for any integer that has a decomposition. Theorems \ref{g-nary_1}, \ref{g-nary_2}, and \ref{g-nary_3} identify three distinct classes of $g$-nary sequences that preserve uniqueness in this way.

We construct a $g$-nary sequence by requiring that the summands are monotonically increasing (starting at 1), setting $A_n = \{0,g\}$ for some constant $g$, allowing a number to be in a given bin at most once, and at each step taking the smallest number that preserves uniqueness. The resulting $g$-nary sequence is well-defined if and only if after computing $n$ numbers of the sequence, we can find an $(n+1)^{\rm{st}}$ number which satisfies the constraints of our construction (most importantly, uniqueness). For simplicity we begin with $g = 2$, constant bin size 3 ($b_n = 3$), and no adjacency condition ($a=0$). Let $I_{n}$ be the set of all legally decomposable numbers using bins $b_1$ through $b_n$. Let $G_{n,j}$ be the gap between the $(j-1)$\textsuperscript{{\rm st}} summand and $j$\textsuperscript{{\rm th}} summand in the $n$\textsuperscript{th} bin, and $\Omega_n$ be the largest number legally representable using only elements from the first $n$ bins. Then we have the following.

\begin{theorem} \label{g-nary_1} For $b_n = 3$, $A_n = \{0,2\}$ and $a=0$, the resulting $g$-nary sequence is well-defined and we have $G_{n,j} > \Omega_{n-1}$.
\end{theorem}

Note that while the gap between adjacent summands in the bin can differ, to keep uniqueness we need the gap between any two adjacent summands in bin $n$ to be larger than the maximum decomposition using all the bins $n-1$.



\begin{proof} We begin with the base case. The first two intervals are
\begin{equation} \underbracket{\ 1,\ 2,\ 3 \ }_{{\rm Bin\ 1}}\ ,\ \underbracket{\    3,\  x, \ y \ }_{{\rm Bin \ 2}}\ ,\ \ldots\end{equation}  with $x < y$; this is due to our requirement that the sequence is monotonically increasing and a number is in a bin at most once. A straightforward calculation shows that the first combination of $x$ and $y$ for which we retain uniqueness is $x = 9$ and $y = 15$. For more details on computing the sequence see Appendix \ref{sec:gnaryterms}.

Now suppose that we retain uniqueness with bins $b_1$ through $b_k$ and for all $n$ such that $1 \leq  n <  k+1$, $G_{n,j}  >  \Omega_{n-1}$. Now we seek to show that we retain uniqueness with bins $b_1$ through $b_{k+1}$ and $G_{k+1,j}  >  \Omega_{k}$.
We have \begin{equation} \underbracket{\ 1,\ 2,\ 3 \ }_{{\rm Bin\ 1}}\ ,\ \underbracket{\    3,\  9, \ 15 \ }_{{\rm Bin \ 2}}\ ,\ \underbracket{\    15,\  45, \ 75 \ }_{{\rm Bin \ 3}}\ ,\ \underbracket{\    75,\  225, \ 375 \ }_{{\rm Bin \ 4}}\ ,\ \dots\underbracket{\ a,\ b, \ c \ }_{{\rm Bin\ k+1}}.\end{equation}
Thus if $G_{k+1, j} \leq  \Omega_k $, $\exists D_1, D_2 \in I_n$, $D_2>D_1$ such that
\begin{equation}
b+c+D_1\ = \ a+c+D_2,
\end{equation}
because by construction $D_2-D_1 \in \{1,2,...,\Omega_k\}$.
Thus we lose uniqueness.
However, if $G_{k+1,j} >  \Omega_k$, there does not exist a combination of $D_1, D_2$ such that $a+b+D_1 =  a+c+D_2$, nor $b+c+D_1 =  a+c+D_2$. Therefore, we keep uniqueness and
\begin{equation}
G_{k+1, j}\ >  \Omega_{k}.
\end{equation}
By induction we keep uniqueness and we have $G_{n+1, j} >  \Omega_{n}$ for all $n$, so this $g$-nary sequence is well-defined.
\end{proof}

\begin{thm} \label{g-nary_2} For $b_n = k$, $A_n = \{0,g\}$ for a pair of constants $g\in \{1,k-1\}$, and $a=0$, the resulting $g$-nary sequence is well-defined and we have $G_{n,j} > \Omega_{n-1}$.
\end{thm}
\begin{proof} We begin with the base case.
The first two intervals are
\begin{equation} \underbracket{\ 1,\ 2,\ \dots , \ a_k \ }_{{\rm Bin\ 1}}\ ,\ \underbracket{\    a_k,\  a_k+\Omega_{1}+1, \ a_k+2\Omega_{1}+1,\ \dots , \ a_k+k\Omega_{1}+1 }_{{\rm Bin \ 2}}\ ,\ \ldots  ;
\end{equation}
this is due to the fact that if $G_{2, j} \leq  \Omega_1 $, $\exists D_1, D_2 \in I_1$, $D_2>D_1$ and $ p \in \{2,3,\dots, k-g+1\}$ such that
\begin{equation}
\sum_{i=p-1}^{p+g-3}a_i+a_{p+g-1}+D_1\ = \sum_{i=p}^{p+g-1}a_i+D_2,
\end{equation}
because by construction $D_2-D_1 \in \{1,2,...,\Omega_{1}\}$.

Now suppose that we retain uniqueness with bins $b_1$ through $b_n$ and for all $1 <  n <  k+1$, $G_{n,j}  >  \Omega_{n-1}$. Now we seek to show that we retain uniqueness with bins $b_1$ through $b_{k+1}$ and $G_{k+1,j}  >  \Omega_{k}$.
We have \begin{equation} \underbracket{\ 1,\ 2,\ \dots , \ a_k \ }_{{\rm Bin\ 1}}\ ,\ \underbracket{\    a_k,\  a_k+\Omega_{1}+1, \ a_k+2\Omega_{1}+1,\ \dots , \ a_k+k\Omega_{1}+1 }_{{\rm Bin \ 2}}\ ,\ \dots\underbracket{\ x_1,\ x_2, \ x_3, \ \dots , \ x_k }_{{\rm Bin\ k+1}}.\end{equation}
Thus if $G_{k+1, j} \leq  \Omega_k $, $\exists D_3, D_4 \in I_n$, $D_4>D_3$ and $ p \in \{2,3,\dots, k-g+1\}$ such that
\begin{equation}
\ \sum_{i=p-1}^{p+g-3}x_i+x_{p+g-1}+D_3\ = \sum_{i=p}^{p+g-1}x_i+D_4,
\end{equation}
because by construction $D_4-D_3 \in \{1,2,...,\Omega_k\}$.
Thus we lose uniqueness.
However, if $G_{k+1,j} >  \Omega_k$, it is clear that there does not exist a linear combination of $D_3, D_4$ and $x_i$'s such that we lose uniqueness.

By induction we keep uniqueness and we have $G_{n+1, j} >  \Omega_{n}$ for all $n$, so this class of $g$-nary sequences is well-defined.
\end{proof}

\begin{thm} \label{g-nary_3} For $b_n = b_l$, $A_n = \{0,g\}$ for some constant $g<b_l$ for all $\ell$ and $a=0$,  the resulting $g$-nary sequence is well-defined and we have $G_{n,j} > \Omega_{n-1}$.
\end{thm}
\begin{proof} We begin with the base case.
The first two intervals are
\begin{equation} \underbracket{\ 1,\ 2,\ \dots , \ a_{b_1} \ }_{{\rm Bin\ 1}}\ ,\ \underbracket{\    a_{b_1},\  a_{b_1}+\Omega_{1}+1, \ a_{b_1}+2\Omega_{1}+1,\ \dots , \ a_{b_1}+b_2\Omega_{1}+1 }_{{\rm Bin \ 2}}\ ,\ \ldots;\end{equation}  this is due to the fact that if $G_{2, j} \leq  \Omega_1 $, $\exists D_1, D_2 \in I_1$, $D_2>D_1$ and $ p \in \{2,3,\dots, b_1-g+1\}$ such that
\begin{equation}
\sum_{i=p-1}^{p+g-3}a_i+a_{p+g-1}+D_1\ = \sum_{i=p}^{p+g-1}a_i+D_2,
\end{equation}
because by construction $D_2-D_1 \in \{1,2,...,\Omega_{1}\}$.

Now suppose that we retain uniqueness with bins $b_1$ through $b_n$ and for all $1 <  n <  k+1$, $G_{n,j}  >  \Omega_{n-1}$. Now we seek to show that we retain uniqueness with bins $b_1$ through $b_{k+1}$ and $G_{k+1,j}  >  \Omega_{k}$.
We have \begin{eqnarray} & & \underbracket{\ 1,\ 2,\ \dots , \ a_{b_1} \ }_{{\rm Bin\ 1}}\ ,\ \underbracket{\    a_{b_1},\  a_{b_1}+\Omega_{1}+1, \ a_{b_1}+2\Omega_{1}+1,\ \dots , \ a_{b_1}+{b_2}\Omega_{1}+1 }_{{\rm Bin \ 2}}\ ,\ \dots\ ,\ \nonumber\\ & & \ \ \ \ \ \ \ \ \dots\ ,\ 
\underbracket{\ x_1,\ x_2, \ x_3, \ \dots , \ x_{b_k+1} }_{{\rm Bin\ k+1}}.\end{eqnarray}

Thus if $G_{k+1, j} \leq  \Omega_k $, then there exist $D_3, D_4 \in I_n$, $D_4>D_3$ and $ q \in \{2,3,\dots, b_{k+1}-g+1\}$ such that
\begin{equation}
\ \sum_{i=q-1}^{q+g-3}x_i+x_{q+g-1}+D_3\ = \sum_{i=q}^{q+g-1}x_i+D_4,
\end{equation}
because by construction $D_4-D_3 \in \{1,2,...,\Omega_k\}$.
Thus we lose uniqueness.
However, if $G_{k+1,j} >  \Omega_k$, it is clear that there does not exist a linear combination of $D_3, D_4$ and $x_i$'s such that we lose uniqueness.
By induction we keep uniqueness and we have $G_{n+1, j} >  \Omega_{n}$ for all $n$, so this class of $g$-nary sequences is well-defined.
\end{proof}

\begin{rem}
Note that $\Omega_{n +1} =  G_{n+1, j}$ for the $g$-nary sequences discussed above.
\end{rem}

\begin{lem}
For $b_n = 3$, $A_n = \{0,2\}$ and $a=0$, there are $4^{n}$ elements in $I_n\ $ for all $n$.
\end{lem}

\begin{proof}
We begin with the base case \begin{equation} \underbracket{\ 1,\ 2,\ 3. \ }_{{\rm Bin\ 1}}\ \end{equation}
There are $4^1$ possible decompositions using only bin $b_1$, yielding the numbers 0, 3, 4 and 5.

Suppose there are $4^n$ elements in $I_n$. We will show that there are $4^{n+1}$ elements in $I_{n+1}$.
\begin{equation}
I_{n+1} \ = 1+ \ \sum_{i=1}^{n+1} \binom{n+1}{i} \binom{3}{2}^i \ = \ 4^{n+1}.
\end{equation}
By induction, there are $4^{n}$ elements in $I_n$ for all $n$.
More generally for different $b_i$ and $g$, we have
\begin{equation}
I_n \ = 1+\ \sum_{i=1}^{n} \binom{n}{i} \binom{b_i}{g}^i.
\end{equation}

Thus, if each $b_i$  equals a constant $b$, then
\begin{equation}
I_n \ = \left ( \binom{b}{g}+1 \right )^n.
\end{equation}
\end{proof}

In the spirit of Theorem \ref{thm:uniquenessnoadjcondition}, a natural question to ask is if one could determine necessary and sufficient conditions on $b_n$ for when a general $g$-nary sequence is well-defined.


\subsection{Tesselations of the Unit Disk}

We end with another candidate to study for a 2-dimensional representation. Consider the tesselation of the unit disk (or upper half plane) by copies of the standard fundamental domain of the modular group ${\rm SL}_2(\mathbb{Z})$; see Figure \ref{fig:sltwoz}. We start by assigning $a_1 = 1$ to the standard fundamental domain, and then introduce an ordering (from the generators $S$ and $T$ of the modular group), with our rule being one cannot use summands from cells that are adjacent under generators of ${\rm SL}_2(\mathbb{Z})$ (or their inverses).

\begin{center}
\begin{figure}[h]
\begin{center}
\scalebox{1}{\includegraphics{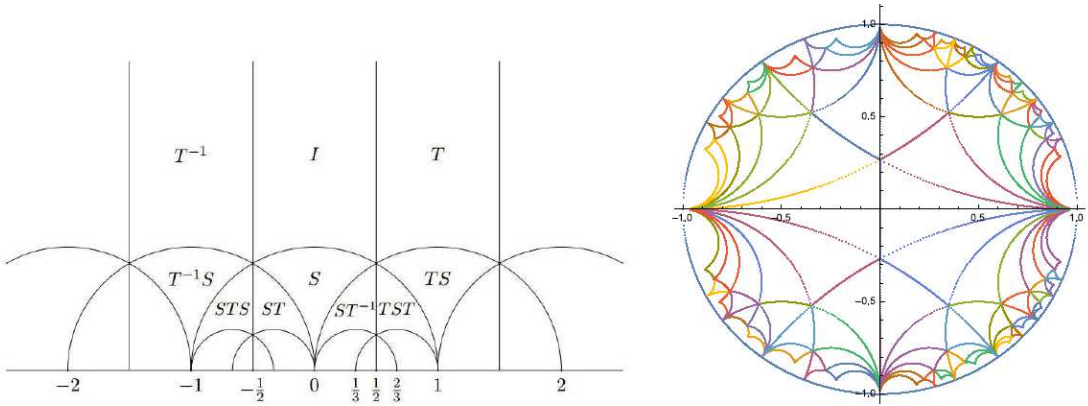}}
\caption{\label{fig:sltwoz} Tesselation of the upper half plane (or unit disk) by copies of the standard fundamental domain of ${\rm SL}_2(\mathbb{Z})$, which is generator by $T$ sending $z$ to $z+1$ and $S$ sending $z$ to $-1/z$. }
\end{center}\end{figure}
\end{center}

\appendix


\section{Computing Terms in a $g$-nary Sequence}\label{sec:gnaryterms}

Here we compute terms in the $g$-nary sequence defined by setting $b_n=3$, $a=0$, and $A_n=\{0,2\}$. As with all $g$-nary sequences, we start with a $1$ in Bin $1$:

\begin{equation} \underbracket{\ 1,\ \_,\ \_ \ }_{{\rm Bin\ 1}}.\end{equation}

The next term must be greater than 1, as the sequence must be monotonic, and a single bin cannot contain two equal numbers. Since no positive number is legally decomposable by this sequence yet (we have to use exactly two terms from any bin), the next term must be 2:

\begin{equation} \underbracket{\ 1,\ 2,\ \_ \ }_{{\rm Bin\ 1}}.\end{equation}

We can now must examine numbers greater than 2, one by one, to see if they preserve uniqueness when added as the third term of the sequence. We see that if we include 3 as the third term, the legal decompositions are $1+2=3$, $1+3=4$, and $2+3=5$, which are all unique, so the third term is 3:

\begin{equation} \underbracket{\ 1,\ 2,\ 3 \ }_{{\rm Bin\ 1}}.\end{equation}

We now must start on Bin 2. Note that while the first term of Bin 2 must be at least as large as 3, it can be equal to 3 because these terms are in separate bins. Note that we cannot use any terms from Bin 2 after adding a single term, since we must use exactly two terms from any bin, so the legal decompositions will remain the same as in the prevoius case. Importantly, this means that we will add 3, the minimal possible number we can put into the bin (since uniqueness is not in question, we simply pick the smallest number maintaining monotonicity):

\begin{equation} \underbracket{\ 1,\ 2,\ 3 \ }_{{\rm Bin\ 1}}\ ,\ \underbracket{\    3,\  \_, \ \_ \ }_{{\rm Bin \ 2}}.\end{equation}

Now, suppose we fill in the remaining slots of Bin 2 with $x$ and $y$:

\begin{equation} \underbracket{\ 1,\ 2,\ 3 \ }_{{\rm Bin\ 1}}\ ,\ \underbracket{\    3,\  x, \ y \ }_{{\rm Bin \ 2}}.\end{equation}

We can choose to use $3$ and $y$ or $x$ and $y$ from our bin. This changes the sum of a decomposition by $x-3$. Thus, $x-3$ must be larger than a change that can be produced by using or not using terms from Bin 1. Since Bin 1 can decompose numbers up to 5, $x-3$ must be larger than 5. Thus, $x-3=6$, so $x=9$. By similar logic, we find that $y-9=6$, so $y=15$. We now have a complete Bin 2:

\begin{equation} \underbracket{\ 1,\ 2,\ 3 \ }_{{\rm Bin\ 1}}\ ,\ \underbracket{\    3,\  9, \ 15 \ }_{{\rm Bin \ 2}}.\end{equation}

We now repeat the logic applied to Bin 2 to Bins 3 and onward. The results of the computation can be seen below for reference:

\begin{equation} \underbracket{\ 1,\ 2,\ 3 \ }_{{\rm Bin\ 1}}\ ,\ \underbracket{\    3,\  9, \ 15 \ }_{{\rm Bin \ 2}}\ ,\ \underbracket{\    15,\  45, \ 75 \ }_{{\rm Bin \ 3}}\ ,\ \underbracket{\    75,\  225, \ 375 \ }_{{\rm Bin \ 4}}\ ,\ \dots\underbracket{\ a,\ b, \ c \ }_{{\rm Bin\ k+1}}.\end{equation}


\ \\

\end{document}